\documentclass[11pt,reqno]{amsart}
\usepackage[utf8]{inputenc}

\usepackage{graphicx} 
\usepackage{amsmath} 
\usepackage{amsthm}  
\usepackage{amssymb} 
\usepackage{siunitx} 
\usepackage{csvsimple} 
\usepackage{tabularx} 
\usepackage{booktabs} 
\usepackage{bbm} 
\usepackage{tikz-cd} 
\usepackage{quiver}
\usepackage{adjustbox}
\usepackage{enumitem}
\usepackage{stmaryrd}
\usepackage[plainpages=false]{hyperref}
\usepackage{faktor} 
\makeatletter 
\DeclareRobustCommand*{\mfaktor}[3][]
{
   { \mathpalette{\mfaktor@impl@}{{#1}{#2}{#3}} }
}
\newcommand*{\mfaktor@impl@}[2]{\mfaktor@impl#1#2}
\newcommand*{\mfaktor@impl}[4]{
   \settoheight{\faktor@zaehlerhoehe}{\ensuremath{#1#2{#3}}}%
   \settoheight{\faktor@nennerhoehe}{\ensuremath{#1#2{#4}}}%
      \raisebox{-0.5\faktor@zaehlerhoehe}{\ensuremath{#1#2{#3}}}%
      \mkern-4mu\diagdown\mkern-5mu%
      \raisebox{0.5\faktor@nennerhoehe}{\ensuremath{#1#2{#4}}}%
}
\makeatother
\counterwithin{equation}{section}

\theoremstyle{definition}
\newtheorem{Definition}{Definition}[section]
\newtheorem*{Construction}{Construction}

\newtheorem{Example}[Definition]{Example}

\theoremstyle{plain}
\newtheorem{Theorem}[Definition]{Theorem}
\newtheorem{Lemma}[Definition]{Lemma}
\newtheorem{Proposition}[Definition]{Proposition}
\newtheorem{Corollary}[Definition]{Corollary}
\newtheorem{Conjecture}[Definition]{Conjecture}
\newtheorem*{Claim}{Claim}

\DeclareMathOperator{\Hom}{Hom}
\DeclareMathOperator{\Ext}{Ext}
\DeclareMathOperator{\ext}{ext}
\DeclareMathOperator{\Der}{Der}

\DeclareMathOperator{\Aut}{Aut}
\DeclareMathOperator{\Coker}{Coker}
\DeclareMathOperator{\Image}{Im}

\DeclareMathOperator{\codim}{codim}

\DeclareMathOperator{\modulecat}{mod}

\DeclareMathOperator{\rank}{rank}
\DeclareMathOperator{\rankv}{\underline{rank}}
\DeclareMathOperator{\dimv}{\underline{dim}}

\DeclareMathOperator{\comp}{Comp}
\DeclareMathOperator{\sgn}{sgn}
\DeclareMathOperator{\diag}{diag}
\DeclareMathOperator{\sub}{sub}
\DeclareMathOperator{\fac}{fac}
\DeclareMathOperator{\codeg}{codeg}

\title{The Aizenbud-Lapid binary operation for symmetrizable Cartan types}
\author{Markus Kleinau}

\begin{document}
\maketitle 

\begin{abstract}
    Aizenbud and Lapid recently introduced a binary operation on the crystal graph $B(-\infty)$ associated to a symmetric Cartan matrix. We extend their construction to symmetrizable Cartan matrices.
\end{abstract}

\setcounter{tocdepth}{1}
\tableofcontents
    
\section{Introduction}


    Let $C$ be a symmetric generalised Cartan matrix, $\Pi$ the preprojective algebra associated to $C$ and $\mathfrak{n}$ the positive part of the symmetric Kac-Moody algebra associated to $C$. The geometry of the representation varieties of $\Pi$ is closely related to the quantum group $U_q(\mathfrak{n})$: In \cite{LusztigSemicanonical}, Lusztig recovered the enveloping algebra $U(\mathfrak{n})$ as an algebra of constructible functions on the representation varieties of $\Pi$. Kashiwara and Saito \cite{KS} showed that the set $\comp$ of irreducible components of those varieties forms the crystal graph $B(-\infty)$ associated to $U_q(\mathfrak{n})$.
    
    Recently Aizenbud, Lapid and Minguez \cite{lapid1}, \cite{lapid2} introduced a binary operation on $\comp$ which is conjectured to recover a similar, but only partially defined, binary operation on $B(-\infty)$. The latter operation was proposed by Leclerc in \cite{imaginary}. Let $\mathbf{B}^*$ be the dual of the canonical basis of $U_q(\mathfrak{n})$. An element $\mathbf{b}\in \mathbf{B}^*$ is called real, if $\mathbf{b}^2 \in q^n\mathbf{B}^*$. Leclerc conjectured the following statement on the multiplicative structure of $\mathbf{B}^*$, which is now a theorem due to \cite{SymmetricProof}. 
    
    \begin{Theorem}[Leclerc's conjecture \cite{imaginary}]\label{LeclercConj}
        Let $\mathbf{b}_1,\mathbf{b}_2\in \mathbf{B}^*$, at least one of which is real. Then there exists a unique $\mathbf{b}'\in \mathbf{B}^*$ such that the expansion of $\mathbf{b}_1\mathbf{b}_2$ is of the form
        \[\mathbf{b}_1\mathbf{b}_2 = q^m\mathbf{b}'+\sum_{\mathbf{c}\neq\mathbf{b}'\in\mathbf{B}^*}\gamma^\mathbf{c}_{\mathbf{b}_1\mathbf{b}_2}(q)\mathbf{c}\]
        where $m\in \mathbb{Z}$ and $\gamma^\mathbf{c}_{\mathbf{b}_1\mathbf{b}_2}\in \mathbb{Z}[q,q^{-1}]$ with $\codeg(\gamma^\mathbf{c}_{\mathbf{b}_1\mathbf{b}_2})>m$.
    \end{Theorem}
    In this case, we set
    \[ \mathbf{b}_1 \diamond \mathbf{b}_2 := \mathbf{b}'.\]
    This construction can be categorified using quiver Hecke algebras, quantum affine algebras or $GL$-representations, see \cite{simplicity}, \cite{lapid2}. An interpretation in terms of cluster algebras is given in \cite{cluster}.

    The binary operation on $\comp$ is based on generic extensions. Let $C_1,C_2\in \comp$ be irreducible components and let $S\subset C_1\times C_2$ be a sufficiently small open subset. Aizenbud and Lapid observed that the closure of the set of extensions of elements of $S$ is another irreducible component that we will call $C_1 * C_2$. Both $\comp$ and $\mathbf{B}^*$ can be parametrised by $B(-\infty)$. Let $\mu: \comp \rightarrow \mathbf{B}^*$ send a component to the corresponding basis vector. The principal motivation for studying these operations is that they should agree under $\mu$.
    \begin{Conjecture}
        Let $C_1,C_2\in \comp$. Then
        \[ \mu(C_1*C_2) = \mu(C_1) \diamond \mu(C_2) \]
        when the latter is defined.
    \end{Conjecture}
    Some work towards this conjecture in type A can be found in \cite{lapid2}.
    
    Now let $C$ be a symmetrizable generalized Cartan matrix. Leclerc's conjecture remains open in this generality. The first goal of this paper is to establish an analogue of the binary operation on $\comp$ in this case. To that end, we need a generalized class of preprojective algebras associated to symmetrizable generalized Cartan matrices. Such a class was introduced in \cite{GLS} by Geiß, Leclerc and Schröer who have already used it to generalize parts of the connections mentioned above, see \cite{GLS4}.

    Although the binary operation $*$ is in general neither associative nor commutative, it does satisfy the following cancellation property.
    \begin{Theorem}
        Let $C_1, C_2,C_2'\in \comp$ be components and assume that $C_1$ is rigid. Then
        \begin{enumerate}[label=(\roman*)]
            \item $C_1 * C_2 = C_1 * C_2'$ implies $C_2 = C_2'$,
            \item $C_2 * C_1 = C_2' * C_1'$ implies $C_2 = C_2'$.
        \end{enumerate}
    \end{Theorem}
    Here a component is called rigid if it contains a rigid representation. This is conjectured to be equivalent to the corresponding basis element being real.
    This property was already shown in \cite{lapid1} in the symmetric case. The second goal of the paper is to establish it in the symmetrizable case.
    
    The preprojective algebras introduced in \cite{GLS} depend on a choice of symmetrizer of the Cartan matrix $C$. The binary operation should be independent of this choice. The third goal of the paper is to verify this when $C$ is symmetric. We do not know whether it also holds when $C$ is symmetrizable.
    
    The first part of the paper will recall the basic theory of $GLS$-preprojective algebras. Section 2.1 constructs these algebras as in \cite{GLS}, section 2.2 introduces some important classes of modules and section 2.3 describes the representation varieties from \cite{GLS4}. The second part will focus on adapting the binary operation from \cite{lapid1}. Section 3.1. contains its construction. Then section 3.2 introduces rigid components to compute some examples. The proof of the cancellation property is given in section 3.3. The final section 3.4 proves independence of the symmetrizer in the symmetric case.
\subsection*{Acknowledgement}
    This paper is based on the author's master thesis which was written in Bonn under the supervision of Jan Schröer. We thank him for his help and guidance on this topic. We also thank Erez Lapid for pointing out an error in the proof of the cancellation property. 
\section{Preprojective algebras from symmetrizable Cartan matrices}
    We start by recalling the family of preprojective algebras constructed in \cite{GLS}. It is an adaptation of the preprojective algebras of species described by Dlab and Ringel in \cite{DlabRingel}. The main difference is that \cite{GLS} works over an algebraically closed field while \cite{DlabRingel} requires certain field extensions.
    \subsection{The preprojective algebra}
    \begin{Definition}
        Let $Q_0$ be a finite set. A symmetrizable generalized Cartan matrix is an integral matrix $C = (c_{ij})_{i,j\in Q_0}$ such that
        \begin{enumerate}[label=(\roman*)]
            \item $c_{ii} = 2$,
            \item $c_{ij} \leq 0$ for all $i\neq j$,
            \item $c_{ij} \neq 0 \Leftrightarrow c_{ji}\neq 0 $,
            \item there is a diagonal integer matrix $D = \diag(c_i)_{i\in Q_0}$ such that $c_i\geq 1$ and $DC$ is symmetric.
        \end{enumerate}
        Any such matrix $D$ is called a symmetrizer of $C$. A symmetrizer is minimal if the sum of its entries is minimal. For the rest of the paper, all Cartan matrices will implicitly be symmetrizable generalized Cartan matrices.
    \end{Definition}
    \begin{Definition}
        An orientation of a Cartan matrix $C$ is a subset $\Omega \subset Q_0\times Q_0$ such that
        \begin{enumerate}[label=(\roman*)]
            \item $\{(i,j),(j,i)\} \cap \Omega \neq \varnothing\Leftrightarrow c_{ij} < 0$,
            \item the quiver $(Q_0,\Omega)$ is acyclic.
        \end{enumerate}
    \end{Definition}
    The definition of the preprojective algebra requires a choice of orientation. Up to isomorphism, the algebra does not depend on this choice and it will not affect any later considerations. 
    
    We fix a Cartan matrix $C$ with symmetrizer $D$ and orientation $\Omega$ for the rest of the paper. Next, we define the preprojective algebra using a quiver with relations.
    \begin{Definition}
        The opposite orientation $\Omega^*$ is the set
        \[\Omega^* = \{(j,i) \mid (i,j) \in \Omega\}.\]
        We set $\overleftrightarrow{\Omega} = \Omega\cup\Omega^*$ and for all $c_{ij}<0$
        \begin{align*}
            &g_{ij}= \gcd(c_{ij},c_{ji})&f_{ij}=|c_{ij}|/g_{ij}.
        \end{align*}
        The double quiver $\overleftrightarrow{Q}(C) = (Q_0,\overleftrightarrow{Q}_1,s,t)$ is given by
        \begin{align*}
            Q_0 &= Q_0\\
            \overleftrightarrow{Q}_1 &= \{\alpha_{ij}^{(g)} \mid (i,j)\in \overleftrightarrow{\Omega}, 1\leq g\leq g_{ij} \}\cup \{\varepsilon_i \mid i \in Q_0\}\\
            s(\alpha_{ij}^{(g)}) &= j\\
            t(\alpha_{ij}^{(g)}) &= i\\
            s(\varepsilon_i) &= t(\varepsilon_i) = i.
        \end{align*}
        The sign of an element of $\overleftrightarrow{\Omega}$ is given by $\sgn(i,j)=1$ and $\sgn(j,i)=-1$ for $(i,j)\in \Omega$.
    \end{Definition}
    \begin{Definition}[Preprojective algebra]
        The preprojective algebra is defined as
        \[\Pi = \Pi(C,D,\Omega)= K\overleftrightarrow{Q}/I\]
        where $I$ is generated by the following relations:
        \begin{enumerate}[label=(\roman*)]
            \item For each $i\in Q_0$ the nilpotency relation \[\varepsilon_i^{c_i} = 0.\]
            \item For each $(i,j)\in \overleftrightarrow{\Omega}$ and each $1\leq g \leq g_{ij}$ the commutativity relation
            \[\varepsilon_i^{f_{ji}}\alpha_{ij}^{(g)} = \alpha_{ij}^{(g)} \varepsilon_j^{f_{ij}}.\]
            \item For each $i\in Q_0$ the mesh relation \[\sum_{\substack{j\in Q_0\\c_{ij}<0}}\sum_{g=1}^{g_{ji}}\sum_{f=0}^{f_{ji}-1} \sgn(i,j)\varepsilon_i^f\alpha_{ij}^{(g)}\alpha_{ji}^{(g)}\varepsilon_i^{f_{ji}-1-f} = 0.\]
        \end{enumerate}
    \end{Definition}
    If $C$ is symmetric and $D= \diag(1,\dots,1)$ is minimal then this definition recovers the preprojective algebra of a path algebra.
    \begin{Example}[$B_2$]
        Let $C$ be the Cartan matrix 
        \[\begin{pmatrix}
            2 & -2 \\
            -1 & 2
        \end{pmatrix}.\]
        The double quiver $\overleftrightarrow{Q}(C)$ is the quiver
        \[\begin{tikzcd}
        	1 & 2.
        	\arrow["{\alpha_{12}}", curve={height=-6pt}, from=1-2, to=1-1]
        	\arrow["{\alpha_{21}}", curve={height=-6pt}, from=1-1, to=1-2]
            \arrow["{\varepsilon_1}"',loop, in=110, out=70 ,distance=1em, from=1-1, to=1-1]
            \arrow["{\varepsilon_2}"',loop, in=110, out=70 ,distance=1em, from=1-2, to=1-2]
        \end{tikzcd}\]
        Let $D = \diag(1,2)$ be the minimal symmetrizer. Then the preprojective algebra is given by 
        \[\Pi = K\overleftrightarrow{Q}(C)/(\varepsilon_1,\varepsilon_2^2,\alpha_{12}\alpha_{21},\varepsilon_2\alpha_{21}\alpha_{12} + \alpha_{21}\alpha_{12}\varepsilon_2).\]
    \end{Example}
    \subsection{Representations of the preprojective algebra}
    We consider a vertex $i$ of $\overleftrightarrow{Q}(C)$. Let $H_i$ be the subalgebra of $e_i\Pi e_i$ generated by $\varepsilon_i$. Then the nilpotency relation yields $H_i \cong K[X]/(X^{c_i})$. These truncated polynomial rings replace the field extensions used by Dlab and Ringel in \cite{DlabRingel}. The main technical challenge arises from the fact that the $H_i$ are not fields. This requires a restriction to well behaved classes of modules. We consider a $\Pi$-module $M$ and define $M_i := e_iM$. As $\varepsilon_i$ is a loop at vertex $i$, its action restricts to a map $M_{\varepsilon_i}: M_i \rightarrow M_i$. This turns $M_i$ into an $H_i$-module.
    \begin{Definition}
        A $\Pi$-module $M$ is locally free if $M_i$ is a free $H_i$-module for all $i\in Q_0$.
    \end{Definition}
    \begin{Definition}
        For $i\in Q_0$ the generalized simple module $E_i$ is the free rank one $H_i$-module, interpreted as a $\Pi$-module. A $\Pi$-module $M$ is called $E$-filtered if it admits a filtration 
        \[0=M_0\subset M_1\subset \dots\subset M_{n-1}\subset M_n=M\]
        where each quotient $\faktor{M_{k+1}}{M_k}$ is isomorphic to some $E_i$.
    \end{Definition}
    The class of locally free modules is closed under extensions, kernels of epimorphisms and cokernels of monomorphisms by \cite[Lemma 3.8]{GLS}. In particular, every $E$-filtered module is locally free. The class of $E$-filtered modules is only closed under extensions. We need a third class, the crystal modules from \cite{GLS4}, to compensate for the other two missing closure properties.
    \begin{Definition}
        Let $M$ be a locally free $\Pi$-module and $i\in Q_0$. The submodule $\sub_i(M)$ is the largest submodule supported only at $i$. Dually, the factor module $\fac_i(M)$ is the largest factor module supported only at $i$. We define the submodule $K_i(M)$ and the factor module $Q_i(M)$ using the canonical short exact sequences
        \[0\rightarrow K_i(M)\rightarrow M\rightarrow \fac_i(M)\rightarrow 0,\]
        \[0\rightarrow \sub_i(M)\rightarrow M \rightarrow Q_i(M)\rightarrow 0.\]
    \end{Definition}
    \begin{Definition}
        Crystal modules are defined inductively. The zero module is a crystal module. A $\Pi$-module $M$ is a crystal module if it is $E$-filtered and for all $i\in Q_0$ the modules $\sub_i(M)$ and $\fac_i(M)$ are locally free while the modules $Q_i(M)$ and  $K_i(M)$ are crystal modules.
    \end{Definition}
    By definition, we have
    \[\textnormal{crystal}\Rightarrow E\textnormal{-filtered}\Rightarrow \textnormal{locally free}.\]
    If $C$ is symmetric and $D$ is minimal then all modules are locally free and the properties '$E$-filtered' and 'crystal' are equivalent. In that case '$E$-filtered' is usually called 'nilpotent'.
    \begin{Proposition}\label{Dyn:crystalCokernel}
        Let $f: M\rightarrow N$ be a monomorphism between crystal modules. Then $\Coker(f)$ is $E$-filtered.
    \end{Proposition}
    \begin{proof}
        Without loss of generality, $M$ is a submodule of $N$. In this case, $\Coker(f) = \faktor{N}{M}$. We will prove the proposition by induction on the dimension of $N$. If $M$ is the zero module then the claim is clear. Assume $M$ is not zero. Let $i\in Q_0$ be a vertex such that $\sub_i(M) \neq 0$. Such an $i$ must exist because $M$ is $E$-filtered. By definition, we have
        \begin{equation}\label{Dyn:subIntersection}
            \sub_i(M) = \sub_i(N)\cap M.
        \end{equation}
        We consider the following filtration of $\faktor{N}{M}$:
        \[0 \subset \faktor{(\sub_i(N)+M) }{M} \subset \faktor{N}{M}.\]
        The first subquotient is
        \[\faktor{(\sub_i(N)+M) }{M} \overset{\eqref{Dyn:subIntersection}}{\cong} \faktor{\sub_i(N)}{\sub_i(M)}.\]
        This is a quotient of free $K[\varepsilon_i]$-modules hence it is a free  $K[\varepsilon_i]$-module, too. As a $\Pi$-module this implies that it is $E$-filtered. The second subquotient is given by 
        \[\faktor{N}{(\sub_i(N)+M)}.\]
        We describe it using the following sublattice of the submodule lattice of $N$:
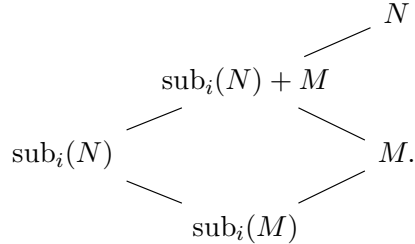
        \[\begin{tikzcd}[column sep = 0.8em,row sep = 0.8em]
        	&& N \\
        	& {\sub_i(N)+M} \\
        	{\sub_i(N)} && M. \\
        	& {\sub_i(M)}
        	\arrow[no head, from=2-2, to=3-1]
        	\arrow[no head, from=2-2, to=3-3]
        	\arrow[no head, from=3-3, to=4-2]
        	\arrow[no head, from=3-1, to=4-2]
        	\arrow[no head, from=1-3, to=2-2]
        \end{tikzcd}\]
        We use the diagram to rewrite the second subquotient as
        \[\faktor{N}{(\sub_i(N)+M)} \cong \left(\faktor{N}{\sub_i(N)}\right)/\left(\faktor{(\sub_i(N)+M)}{\sub_i(N)}\right)\]
        where
        \[\faktor{N}{\sub_i(N)} = Q_i(N)\]
        and
        \[\faktor{(\sub_i(N)+M)}{\sub_i(N)} \cong \faktor{M}{\sub_i(M)} = Q_i(M).\]
        Therefore, the second subquotient is a quotient of crystal modules of lower dimension, hence $E$-filtered by induction. Now $\faktor{N}{M}$ has a filtration with $E$-filtered subquotients, so it is $E$-filtered.
    \end{proof}
    The dual statement is the following:
    \begin{Proposition}\label{Dyn:crystalKernel}
        Let $f:M\rightarrow N$ be an epimorphism between crystal modules. Then $\ker(f)$ is $E$-filtered.
    \end{Proposition}
    Many homological properties of the classical preprojective algebras can be generalized. This requires restricting to locally free modules.
    \begin{Definition}
       We define bilinear functions $\alpha,\beta : \mathbb{Z}^{Q_0}\times\mathbb{Z}^{Q_0}\rightarrow \mathbb{Z}$ by
        \[\alpha (\mathbf{d},\mathbf{e}) = \sum_{i\in Q_0} c_i\mathbf{d}_i \mathbf{e}_i\]
        \[\beta (\mathbf{d},\mathbf{e}) = \sum_{(i,j)\in \Omega} c_i|c_{ij}|\mathbf{d}_{i} \mathbf{e}_{j}.\]
        The symmetrized Euler form $(-,?)$ is given by
        \[(\mathbf{d},\mathbf{e}) = \alpha(\mathbf{d},\mathbf{e}) +\alpha(\mathbf{e},\mathbf{d}) - \beta(\mathbf{d},\mathbf{e}) - \beta(\mathbf{e},\mathbf{d}).\]
    \end{Definition}
    The homological bilinear form for path algebras uses dimension vectors. For locally free $\Pi$-modules the dimension has to be replaced with the rank of the $H_i$-modules.
    \begin{Definition}
        Let $M$ be a locally free $\Pi$-module. The rank vector of $M$ is the integral vector 
        \[\rankv(M) = (\rank_{H_i}(M_i))_{i\in Q_0}\in \mathbb{Z}^{Q_0}.\]
        In particular, $\dimv(M)_i = c_i\rankv(M)_i$.
    \end{Definition}
    \begin{Theorem}
        Let $M,N$ be locally free $\Pi$-modules with rank vectors $\rankv (M) = \mathbf{d}$ and $\rankv (N) = \mathbf{e}$. Then we have
        \begin{enumerate}[label=(\roman*)]
            \item the Ext-formula
                \begin{equation}\label{Dynext:formula}
                    \dim \Hom_\Pi(M,N) - \dim \Ext^1_\Pi(M,N) + \dim \Hom_\Pi(N,M) = (\mathbf{d},\mathbf{e}),
                \end{equation}
            \item the Ext-duality
                \begin{equation}\label{Dynext:symmetry}
                    \dim \Ext^1_\Pi(M,N) = \dim \Ext^1_\Pi(N,M),
                \end{equation}
        \end{enumerate}
    \end{Theorem}   
    \begin{proof}
       This is Theorem 12.6 in \cite{GLS}.
    \end{proof}
    \subsection{Representation varieties of preprojective algebras}
    The classes introduced earlier are closely related to the representation varieties of $\Pi$. This section is based on \cite{GLS4}. First, we need to replace the $Q_0$-graded vector spaces used in the classical case.
    \begin{Definition}
        Let $T$ be the subalgebra of $\Pi$ generated by the $e_i$ and the $\varepsilon_i$. So
        \[T \cong \prod_{i\in Q_0} H_i.\]
    \end{Definition}
    We will examine the moduli space of $\Pi$-module structures on a given locally free $T$-module.
    \begin{Definition}
        Let $V= \bigoplus_{i\in Q_0} V_i$ be a locally free $T$-module. Consider the vector space
        \[\prod_{\alpha_{ij}^{(g)}\in \overleftrightarrow{Q}_1} \Hom_{K}(V_{j},V_{i})\]
        interpreted as an affine variety. The representation variety of $\Pi$ on $V$ is the closed subvariety $R_\Pi(V)$ described by the relations of $\Pi$. We will identify elements of $R_\Pi(V)$ with the modules they describe.\\
        The general linear group of $V$ is the group
        \[GL(V) = \prod_{i\in Q_0} \Aut_{H_i}(V_i).\]
        It acts on $R_\Pi(V)$ by conjugation
        \begin{align*}
            GL(V) \times R_\Pi(V)&\rightarrow R_\Pi(V)\\
            \left(\left(g_i\right)_{i\in Q_0},\left(M_{ij}^{(g)}\right)_{\alpha_{ij}^{(g)}\in \overleftrightarrow{Q}_1}\right)& \mapsto \left(g_{i}M_{ij}g_{j}^{-1}\right)_{\alpha_{ij}^{(g)}\in \overleftrightarrow{Q}_1}.
        \end{align*}
        The orbits of this action are the isomorphism classes of locally free $\Pi$-modules.
    \end{Definition}
    The variety of all $\Pi$-representations is too big for our purposes. We will restrict to subvarieties corresponding to the classes of $E$-filtered, respectively crystal modules. In the classical case this corresponds to the restriction to nilpotent modules in \cite{LusztigSemicanonical}.
    \begin{Definition}
        Let $V= \bigoplus_{i\in Q_0} V_i$ be a locally free $T$-module. We define the subvariety of $E$-filtered modules as
        \[ R_\Pi^E(V) = \overline{\{M \in R_\Pi(V)\mid M \textnormal{ is } E\textnormal{-filtered}\}} \]
        and the subvariety of crystal modules as
        \[ R_\Pi^C(V) = \overline{\{M \in R_\Pi(V)\mid M \textnormal{ is a crystal module}\}}. \]
    \end{Definition}
    Without taking closures, these sets would only be constructible. Both of these subvarieties are $GL(V)$-stable and we clearly have
    \[ R_\Pi^C(V) \subset R_\Pi^E(V) \subset R_\Pi(V). \]
    The relation between the first two varieties can be described more precisely.
    \begin{Proposition}\label{Dyn:CrystalComps}
        The variety $R^C_\Pi(V)$ is the union of the irreducible components of $R^E_\Pi(V)$ of maximum dimension. In particular, it is equidimensional.
    \end{Proposition}
    \begin{proof}
        This is Proposition 4.4 in \cite{GLS4}.
    \end{proof}
    We can now introduce $\comp$, the set on which the binary operation will be defined.
    \begin{Definition}
        Let $V$ be a locally free $T$-module with rank vector $\mathbf{d}$. Let $\comp(V)$ be the set of irreducible components of $R^C_\Pi(V)$. We set $\comp(\mathbf{d}) = \comp(V)$. Up to a canonical bijection this does not depend on $V$. Finally, we set
        \[\comp = \bigcup_{\mathbf{d} \in \mathbb{N}_0^{Q_0}}\comp(\mathbf{d}).\]
    \end{Definition}
    The set $\comp$ provides a geometric model of the crystal graph $B(-\infty)$, see \cite{GLS4}.
    A complete description of the case $B_2$ with minimal symmetrizer and rank vector $(2,1)$ can be found in section 8.2 of \cite{GLS4}. 
    The next results will help to compute dimensions of various spaces later.
    \begin{Proposition}
        Let $V$ and $V'$ be locally free $T$-modules with rank vectors $\rankv(V) = \mathbf{d}$ and $\rankv(V') = \mathbf{e}$. Then we have
        \begin{enumerate}[label=(\roman*)]
            \item \begin{equation}\label{dimRepVar}
                \dim R^C_\Pi(V) = \beta(\mathbf{d},\mathbf{d}),
            \end{equation}
            \item \begin{equation}\label{dimHomSpace}
                \dim \Hom_T(V,V') = \alpha(\mathbf{d},\mathbf{e}),
            \end{equation}
            \item \begin{equation}\label{dimGL}
                \dim GL(V) = \alpha(\mathbf{d},\mathbf{d}).
            \end{equation}
        \end{enumerate}
    \end{Proposition}
    \begin{proof}
        $(i)$ is Corollary 4.5 in \cite{GLS4} while $(ii)$ and $(iii)$ follow from the representation theory of $K[X]/(X^{c_i})$.
    \end{proof}
    \begin{Proposition}\label{Dyn:semicont}
        Let $V$ and $V'$ be locally free $T$-modules. Then the functions 
        \begin{align*}
            R_\Pi(V)\times R_\Pi(V') &\rightarrow \mathbb{Z}\\
            (M,N)&\mapsto \dim \Hom_\Pi(M,N),\\
            R_\Pi(V)\times R_\Pi(V') &\rightarrow \mathbb{Z}\\
            (M,N)&\mapsto \dim \Ext^1_\Pi(M,N)
        \end{align*}
        are upper semicontinuous.
    \end{Proposition}
    \begin{proof}
        This is a special case of Lemma 4.3 in \cite{BasicsModuleVarieties}.
    \end{proof}
    \begin{Proposition}
        Let $M$ and $N$ be locally free $\Pi$-modules. Then there is a short exact sequence
        \begin{equation}\label{Dyn:goodcomplex}
            0 \rightarrow \Hom_\Pi(M,N) \rightarrow \Hom_T(M,N)\rightarrow \Der_\Pi(M,N) \rightarrow \Ext^1_\Pi(M,N) \rightarrow 0.
        \end{equation}
        Here, $\Der_\Pi(M,N)$ is the set of derivations which send all $e_i$ and all $\varepsilon_i$ to 0.
    \end{Proposition}
\section{The binary operation \texorpdfstring{$*$}{*}}
    Generic extensions provide a binary operation on $\comp$. This chapter is based on \cite{lapid1} where the operation was introduced and studied for classical preprojective algebras. Most results and proofs in this chapter are direct generalizations to the symmetrizable case.
    \subsection{Construction}
    We first need an appropriate notion of generic modules.
    \begin{Definition}
        For each component $C\in \comp$ let $\widehat{C}$ be the interior of the set of crystal modules in $C$. By construction, this is a dense open subset of $C$.\\
        For $C_1, C_2 \in \comp$,  we define the generic values
        $\hom_\Pi(C_1,C_2)$ and $\ext^1_\Pi(C_1,C_2)$ as
        \[\hom_\Pi(C_1,C_2) = \min\{\dim \Hom_\Pi(M_1,M_2) \mid M_1 \in C_1,M_2 \in C_2\}\]
        \[\ext^1_\Pi(C_1,C_2) = \min\{\dim \Ext^1_\Pi(M_1,M_2) \mid M_1 \in C_1,M_2 \in C_2\}\]
        and the set of generic pairs as
        \[S(C_1,C_2) = \{(M_1,M_2) \in \widehat{C}_1 \times \widehat{C}_2\ \mid \dim  \Ext^1_\Pi(M_1,M_2) = \ext^1_\Pi(C_1,C_2)\}.\]
    \end{Definition}
    The set $S(C_1,C_2)$ is open because $\dim \Ext^1_\Pi(-,?)$ is upper semicontinuous by Proposition \ref{Dyn:semicont}. Elements of $S(C_1,C_2)$ also achieve the generic values for $\hom$.
    \begin{Lemma}
        Let $(M_1,M_2) \in S(C_1,C_2)$. Then we also have
        \[ \dim \Hom_\Pi(M_1,M_2) = \hom_\Pi(C_1,C_2),\]
        \[ \dim \Hom_\Pi(M_2,M_1) = \hom_\Pi(C_2,C_1).\]
    \end{Lemma}
    \begin{proof}
        This is Lemma 4.4 in \cite{BasicsModuleVarieties} combined with the Ext-symmetry \eqref{Dynext:symmetry}.
    \end{proof}
    Let $C_1 \in \comp(\mathbf{d}_1)$ and $C_2 \in \comp(\mathbf{d}_2)$. Applying Ext-symmetry \eqref{Dynext:symmetry} and the Ext-formula \eqref{Dynext:formula} to an element of $S(C_1,C_2)$ gives the corresponding statements for components:
    \begin{equation}\label{binary:CompExtDual}
        \ext_\Pi(C_1,C_2) = \ext_\Pi(C_2,C_1)
    \end{equation}
    \begin{equation}\label{binary:CompExtEuler}
        \hom_\Pi(C_1,C_2) - \ext^1_\Pi(C_1,C_2) + \hom_\Pi(C_2,C_1) = (\mathbf{d}_1,\mathbf{d}_2).
    \end{equation}
    We will study the extensions of the generic pairs of modules in $S(C_1,C_2)$.
    \begin{Definition}
        Let $V = V_1 \oplus V_2$ be the direct sum of two locally free $T$-modules and let $C_1 \in \comp(V_1)$, $C_2\in \comp(V_2)$ and $S\subseteq C_1 \times C_2$ open. Then the set of generic extensions of $S$ is
        \[\mathcal{E} (S) = \{M \in R_\Pi(V) \mid \exists \  0\rightarrow M_2\rightarrow M \rightarrow M_1 \rightarrow 0 \textnormal{ exact with } (M_1,M_2) \in S\}.\]
        We set $C_1 * C_2 = \overline{\mathcal{E}(S(C_1,C_2))}.$
    \end{Definition}
    The set $S$ will usually be a subset of $S(C_1,C_2)$. The class of $E$-filtered modules is closed under extension so the set $\mathcal{E} (S)$ lies in $R^E_\Pi(V)$. The variety $C_1 * C_2$ turns out to be an irreducible component of $R^C_\Pi(V)$, turning $*$ into a binary operation on $\comp$. This is a direct generalisation of Theorem 3.1 in \cite{lapid1}. We adapt their proof.
    \begin{Theorem}\label{binary:wellDefined}
        In the situation above $C_1 * C_2 \in \comp(V)$.\\
        In addition, for any non-empty open $S \subseteq S(C_1,C_2)$ we have $\overline{\mathcal{E}(S)} = C_1 * C_2$. 
    \end{Theorem}
    \begin{proof}
          First, we need to translate the construction of $C_1 * C_2$ into algebraic geometry. We set $\rankv V_1 = \mathbf{d}_1$ and $\rankv V_2 = \mathbf{d}_2$ and consider the variety of derivations over $S$:
         \[Z = \{(M_1,M_2,d) \mid (M_1,M_2)\in S, d \in \Der_\Pi(M_1,M_2)\}.\]
         We can consider $Z$ as a subvariety of $R^E_\Pi(V)$ by identifying $(M_1,M_2,d)$ with the module defined by $d$.
         We apply the exact sequence \eqref{Dyn:goodcomplex} to a pair $(M_1,M_2)\in S(C_1,C_2)$:
         \begin{equation}\label{binary:DerivationSequence}
             0 \rightarrow \Hom_\Pi(M_1,M_2) \rightarrow \Hom_T(M_1,M_2)\rightarrow \Der_\Pi(M_1,M_2) \rightarrow \Ext^1_\Pi(M_1,M_2) \rightarrow 0.
         \end{equation}
         The dimensions of all but the third term are constant on $S$. Then the dimension of $\Der_\Pi(M_1,M_2)$ must also be constant on $S$. This implies that $Z$ is a vector bundle over $S$, in particular $Z$ is irreducible, too. Every extension between two modules is isomorphic to a module defined by a derivation. This shows $\mathcal{E}(S) = GL(V).Z$. Using that $GL(V)$ is irreducible, we get that $\overline{\mathcal{E}(S)}$ is irreducible. It remains to compute its dimension. We identify $\Hom_T(V_2,V_1)$ with a subgroup of $GL(V)$ via $L \mapsto \left(\begin{smallmatrix}
                                    1 & L \\
                                    0 & 1
                                \end{smallmatrix}\right)$
        and call this subgroup $R$. We consider the action map 
        \[ F:R \times Z \rightarrow C_1 * C_2.\] Its image is a subset of $\mathcal{E}(S)$. In order to show that this image has a high dimension we need to exhibit a small fibre. Let us fix $M = (M_1,M_2,0)\in Z$. It corresponds to the direct sum $M_1 \oplus M_2 \in R^E_\Pi(V)$. We need to show that the fibre $(R \times Z) _M$ is sufficiently small. In fact it is given by $\Hom_\Pi(M_2,M_1)\times\{M\}$: Writing out $L.M = N$ as block matrices, the lower left corner and diagonal show $M = N$ while the top right corner shows $L \in \Hom_\Pi(M_2,M_1)$. Turning the previous discussion into formulas for the dimensions yields:
        \begin{Claim}
            \begin{enumerate}[label=(\roman*)]
                \item $\dim  S = \beta(\mathbf{d}_1,\mathbf{d}_1) + \beta(\mathbf{d}_2,\mathbf{d}_2)$,
                \item $\dim Z =  \alpha (\mathbf{d}_1,\mathbf{d}_2) + \ext^1_\Pi(C_1,C_2) -\hom_\Pi(C_1,C_2) + \dim S$,
                \item $\dim R = \alpha (\mathbf{d}_2,\mathbf{d}_1)$,
                \item $\dim R \times Z = \hom_\Pi(C_2,C_1) + \beta (\mathbf{d}_1+\mathbf{d}_2,\mathbf{d}_1+\mathbf{d}_2)$,
                \item $\dim (R\times Z)_M = \hom_\Pi(C_2,C_1)$,
                \item $\dim \mathcal{E}(S) \geq \dim R\times Z - \dim (R\times Z)_M = \beta (\mathbf{d}_1+\mathbf{d}_2,\mathbf{d}_1+\mathbf{d}_2)$.
            \end{enumerate}
        \end{Claim}
        \begin{proof}\renewcommand{\qedsymbol}{}
            \begin{enumerate}[label=(\roman*)]
                \item This holds because $S$ is an open subset of $C_1 \times C_2$ and $\dim C_i = \beta (\mathbf{d}_i,\mathbf{d}_i)$ by \eqref{dimRepVar}.
                \item The exact sequence \eqref{binary:DerivationSequence} gives the rank of $Z$ as a vector bundle over $S$: \[\rank_S Z =\dim \Hom_T(V_1,V_2) + \ext^1_\Pi(C_1,C_2) -\hom_\Pi(C_1,C_2).\] We know $ \dim \Hom_T(V_1,V_2) = \alpha(\mathbf{d}_1,\mathbf{d}_2)$ from \eqref{dimHomSpace}. Now the claim follows from the fact that the dimension of a vector bundle is the dimension of its base plus its rank.
                \item We defined $R$ by embedding $\Hom_T(V_2,V_1)$ into $GL(V)$. Its dimension is given by $\dim \Hom_T(V_2,V_1) = \alpha (\mathbf{d}_2,\mathbf{d}_1)$, see \eqref{dimHomSpace}.
                \item We compute
                    \begin{align*}
                        \dim R\times Z &= \dim R + \dim Z\\
                        &= \alpha (\mathbf{d}_2,\mathbf{d}_1) + \alpha (\mathbf{d}_1,\mathbf{d}_2) + \ext^1_\Pi(C_1,C_2) -\hom_\Pi(C_1,C_2) + \dim S\\
                        &= \hom_\Pi(C_2,C_1) + \beta(\mathbf{d}_1,\mathbf{d}_2) + \beta(\mathbf{d}_1,\mathbf{d}_2) + \dim S\\
                        &= \hom_\Pi(C_2,C_1) + \beta(\mathbf{d}_1,\mathbf{d}_2) + \beta(\mathbf{d}_2,\mathbf{d}_1) + \beta(\mathbf{d}_1,\mathbf{d}_1) + \beta(\mathbf{d}_2,\mathbf{d}_2)\\
                        &= \hom_\Pi(C_2,C_1) + \beta (\mathbf{d}_1+\mathbf{d}_2,\mathbf{d}_1+\mathbf{d}_2)
                    \end{align*}
                    where the second equality follows from (ii) and (iii), the third from the Ext-formula for components \eqref{binary:CompExtEuler}, the fourth from (i) and the last from bilinearity of $\beta$.
                \item We identified $(R\times Z)_M$ as $\Hom_\Pi(M_2,M_1)\times\{M\}$. Since $(M_1,M_2)\in S(C_1,C_2)$ the claim holds.
                \item We have
                \[\dim \Image (F) \geq \dim R \times Z - \dim (R \times Z)_M = \beta (\mathbf{d}_1+\mathbf{d}_2,\mathbf{d}_1+\mathbf{d}_2)\]
                by (iv) and (v). Since $\Image(F) \subseteq \mathcal{E}(S)$
                this bound also holds for $\mathcal{E}(S)$.
            \end{enumerate}
        \end{proof}
        In conclusion, we get that $\overline{\mathcal{E}(S)}$ is an irreducible subset of dimension
        \[\dim \mathcal{E}(S) \geq  \beta (\mathbf{d}_1+\mathbf{d}_2,\mathbf{d}_1+\mathbf{d}_2) = \dim R^C_\Pi(V),\]
        hence an irreducible component of maximal dimension. Applying this to $S$ and to $S(C_1,C_2)$ and using that $\mathcal{E}(S)\subseteq\mathcal{E}(S(C_1,C_2))$ gives that
        \[\overline{\mathcal{E}(S)} = \overline{\mathcal{E}(S(C_1,C_2))} = C_1 *C_2\]
        is an irreducible component of maximal dimension.
    \end{proof}
    \begin{Corollary}
        The construction $*$ defines a binary operation on $\comp$.
    \end{Corollary}
\subsection{Computation of \texorpdfstring{$*$}{*}}
    To simplify computing this operation we need a way to identify components. First we need to compute the dimensions of orbits.
    \begin{Proposition}
        Let $M \in R^C_\Pi(V)$ be a module. Then \[\codim \overline{\mathcal{O}_M} = \frac{1}{2}\dim \Ext^1_\Pi(M,M).\]
    \end{Proposition}
    \begin{proof}
        Let $\mathbf{d} = \rankv M$. The orbit $\mathcal{O}_M$ is irreducible since $GL(V)$ is. To compute its dimension we need to study the stabilizer at $M$ which is $\Aut_\Pi(M)$. We have
        \[ \dim \Aut_\Pi(M) = \dim \Hom_\Pi(M,M) = \frac{1}{2}(\mathbf{d},\mathbf{d}) + \frac{1}{2}\dim \Ext^1_\Pi(M,M)\]
        where we use the Ext-formula \eqref{binary:CompExtEuler} for the second equality. Now 
        \begin{align*}
            \codim \overline{\mathcal{O}_M} &= \dim R^C_\Pi(V) -\overline{\mathcal{O}_M}\\
            &= \dim R^C_\Pi(V) -(\dim GL(V)- \dim \Aut_\Pi(M))\\ 
            &= \underbrace{\beta(\mathbf{d},\mathbf{d}) - \alpha (\mathbf{d},\mathbf{d}) + \frac{1}{2}(\mathbf{d},\mathbf{d})}_{=0}+ \frac{1}{2}\dim \Ext^1_\Pi(M,M)
        \end{align*}
        using \eqref{dimRepVar} and \eqref{dimGL}.
    \end{proof}
    A module $M$ is called rigid, if $\Ext^1_\Pi(M,M) = 0$.
    \begin{Corollary}
        Let $M$ be a rigid crystal module. Then $\mathcal{O}_M$ is open and $\langle M \rangle := \overline{\mathcal{O}_M}$ is an irreducible component. We call irreducible components of this form rigid. 
    \end{Corollary}
    The next lemma gives a way to compute $*$ for rigid components.
    \begin{Lemma}
        We consider an exact sequence of rigid crystal $\Pi$-modules
        \[0 \rightarrow M_2 \rightarrow M \rightarrow M_1\rightarrow 0.\]
        Then $\langle M_1\rangle * \langle M_2\rangle = \langle M\rangle$.
    \end{Lemma}
    \begin{proof}
        The open subsets $S(\langle M_1\rangle,\langle M_2\rangle)$ and $\mathcal{O}_{M_1} \times \mathcal{O}_{M_2}$ of $\langle M_1\rangle\times\langle M_2\rangle$ intersect because $\langle M_1\rangle\times\langle M_2\rangle$ is irreducible. Then $M$ must be in $\mathcal{E}(S(\langle M_1\rangle,\langle M_2\rangle)) \subseteq \langle M_1\rangle * \langle M_2\rangle$. The lemma follows. 
    \end{proof}
    The binary operation is neither commutative nor associative, even in the simplest cases, as the next example shows.
    \begin{Example}[$A_2$]
        We consider the case of $A_2$ with minimal symmetrizer. Let $S_1$ and $S_2$ be the simple modules. Then the previous lemma shows:
        \[ (\langle S_1\rangle *\langle S_2\rangle)*\langle S_1\rangle = \left\langle\begin{pmatrix} S_1 \\ S_2 \end{pmatrix}\right\rangle * \langle S_1\rangle = \left\langle  \begin{pmatrix} S_1 \\ S_2 \end{pmatrix} \oplus S_1\right\rangle  \]
        \[ \langle S_1\rangle *(\langle S_2\rangle *\langle S_1\rangle) = \langle S_1\rangle * \left\langle\begin{pmatrix} S_2 \\ S_1 \end{pmatrix}\right\rangle  = \left\langle  \begin{pmatrix} S_2 \\ S_1 \end{pmatrix} \oplus S_1\right\rangle.  \]
    \end{Example}
    \begin{Example}[$B_2$]
        Let $C$ be of type $B_2$ and $D$ be minimal. In this case, all components are rigid. The following table shows $\langle M \rangle * \langle N \rangle$ for each pair of non-projective indecomposable rigid $\Pi$-modules. A cell marked with $\oplus$ corresponds to the case $\langle M \rangle * \langle N \rangle = \langle M\oplus N \rangle$.
        \[\arraycolsep=3pt\renewcommand{\arraystretch}{2}\begin{array}{c||c|c|c|c|c|c}
            \mfaktor{M}{N} & \begin{smallmatrix}1\\1\end{smallmatrix} & 2 & \begin{smallmatrix}1\\1\\2\end{smallmatrix} & \begin{smallmatrix}2\\1\\1\end{smallmatrix} & \begin{smallmatrix}2&&\\1&&2\\&1&\end{smallmatrix} & \begin{smallmatrix}&1&\\2&&1\\&&2\end{smallmatrix}\\
            \hline \hline  \begin{smallmatrix}1\\1\end{smallmatrix}&\oplus&\begin{smallmatrix}1\\1\\2\end{smallmatrix}&\oplus&\oplus&\begin{smallmatrix}&1&\\2&&1\\1&&2\\&1&\end{smallmatrix}&\begin{smallmatrix}1\\1\\2\end{smallmatrix}\oplus \begin{smallmatrix}1\\1\\2\end{smallmatrix}\\
             \hline             2&\begin{smallmatrix}2\\1\\1\end{smallmatrix}&\oplus&\begin{smallmatrix}2\\1\\1\\2\end{smallmatrix}&\begin{smallmatrix}2&&\\1&&2\\&1&\end{smallmatrix}&\oplus&\oplus\\
             \hline             \begin{smallmatrix}1\\1\\2\end{smallmatrix}&\oplus&\begin{smallmatrix}&1&\\2&&1\\&&2\end{smallmatrix}&\oplus&\begin{smallmatrix}&1&\\2&&1\\1&&2\\&1&\end{smallmatrix}&\begin{smallmatrix}&1&\\2&&1\\1&&2\\&1&\end{smallmatrix}\oplus2&\oplus\\
             \hline             \begin{smallmatrix}2\\1\\1\end{smallmatrix}&\oplus&\begin{smallmatrix}2\\1\\1\\2\end{smallmatrix}&\begin{smallmatrix}1\\1\end{smallmatrix}\oplus\begin{smallmatrix}2\\1\\1\\2\end{smallmatrix}&\oplus&\oplus&\begin{smallmatrix}2\\1\\1\\2\end{smallmatrix}\oplus\begin{smallmatrix}1\\1\\2\end{smallmatrix}\\
             \hline             \begin{smallmatrix}2&&\\1&&2\\&1&\end{smallmatrix}&\begin{smallmatrix}2\\1\\1\end{smallmatrix}\oplus\begin{smallmatrix}2\\1\\1\end{smallmatrix}&\oplus&\begin{smallmatrix}2\\1\\1\end{smallmatrix}\oplus\begin{smallmatrix}2\\1\\1\\2\end{smallmatrix}&\oplus&\oplus&\begin{smallmatrix}2\\1\\1\\2\end{smallmatrix}\oplus\begin{smallmatrix}2\\1\\1\\2\end{smallmatrix}\\
             \hline             \begin{smallmatrix}&1&\\2&&1\\&&2\end{smallmatrix}&\begin{smallmatrix}&1&\\2&&1\\1&&2\\&1&\end{smallmatrix}&\oplus&\oplus&\begin{smallmatrix}&1&\\2&&1\\1&&2\\&1&\end{smallmatrix}\oplus2&\begin{smallmatrix}&1&\\2&&1\\1&&2\\&1&\end{smallmatrix}\oplus2\oplus2&\oplus\\
        \end{array}\]
    \end{Example}
    \subsection{Cancellation of components}
    The binary operation is defined using generic extensions. Similar constructions exist for the other two modules in a short exact sequence: kernels and cokernels.
    \begin{Definition}
        Let $V = V_1 \oplus V_2$ be the direct sum of two locally free $T$-modules and let $C_1 \in \comp(V_1)$, $C\in \comp(V)$ and $S\subseteq C_1 \times C$ open. Then the set of generic kernels of $S$ is
        \[\mathcal{K} (S) = \{M_2 \in R_\Pi(V_2) \mid \exists \  0\rightarrow M_2\rightarrow M \rightarrow M_1 \rightarrow 0 \textnormal{ exact with } (M_1,M) \in S\}.\]
        We set $\mfaktor{C_1}{C} = \overline{\mathcal{K}(S(C_1,C))}$ if this is an element of $\comp$.\\ \\
        Dually, let $C_2 \in \comp(V_2)$ and $S\subseteq C \times C_2$ open. Then the set of generic cokernels of $S$ is
        \[\mathcal{Q} (S) = \{M_1 \in R_\Pi(V_1) \mid \exists \  0\rightarrow M_2\rightarrow M \rightarrow M_1 \rightarrow 0 \textnormal{ exact with } (M,M_2) \in S\}.\]
        We set $\faktor{C}{C_2} = \overline{\mathcal{Q}(S(C,C_2))}$ if this is an element of $\comp$.
    \end{Definition}
     Proposition \ref{Dyn:crystalKernel} implies that $\mathcal{K}(S(C_1,C))\subset R^E_\Pi(V_2)$ and Proposition \ref{Dyn:crystalCokernel} implies that $\mathcal{Q}(S(C,C_2)) \subset R^E_\Pi(V_1)$. We  will focus on the case of cokernels while only stating the dual results for kernels. These operations are a kind of inverse to $*$ as seen in the following proposition. Its proof is based on the proof of the closely related Proposition 5.1 in \cite{lapid1}.    \begin{Proposition}\label{binary:cancellation}
         Let $V = V_1 \oplus V_2$ be the direct sum of two locally free $T$-modules and let $C_1 \in \comp(V_1)$, $C_2\in \comp(V_2)$ be components. Assume that $C_2$ is rigid. Then
         \[\faktor{(C_1 * C_2)}{C_2} = C_1.\]
    \end{Proposition}
    \begin{proof}
        We set $C=C_1*C_2$ and consider the variety of maps between generic modules in $C_2$ and $C$
        \[\Tilde{X} = \{(M,M_2,\varphi) \mid (M,M_2)\in S(C,C_2),\ \varphi \in \Hom_\Pi(M_2,M)\}.\]
        Since $\dim \Hom_\Pi(M_2,M)$ is constant on $S(C,C_2)$, the variety $\Tilde{X}$ is a vector bundle over $S(C,C_2)$. In particular, it is irreducible. Now we restrict to injections into a complement of $V_1$
        \[ X = \{(M,M_2,\varphi)\in \Tilde{X} \mid \varphi \textnormal{ is injective and } \Image(\varphi)\cap V_1 = 0 \}. \]
        This is an open subset of $\Tilde{X}$ because both conditions are open. However, a priori it might be empty. We will study the map
        \begin{align*}
            F:X &\rightarrow R^E_\Pi(V_1)\\
            (M,M_2,\varphi)&\mapsto \Coker(\varphi).
        \end{align*}
        Here, $\Coker(\varphi)$ induces a module structure on $V_1$ via the isomorphism \[V_1\rightarrow V \rightarrow \Coker(\varphi).\] The image of $F$ lies in $R^E_\Pi(V_1)$ by Proposition \ref{Dyn:crystalCokernel} and it is irreducible because $X$ is. Note that $GL(V).F(X)$ is $Q(S(C,C_2))$. We will need the following claim.
        \begin{Claim}
            $C_1 \subseteq \overline{F(X)}$.
        \end{Claim}
        Using the claim we can show that 
        \[C_1 \subseteq \overline{F(X)} \subseteq \overline{GL(V).F(X)} = \overline{Q(S(C,C_2))} = \faktor{(C_1 *C_2)}{C_2}.\]
        Here, the two inclusions are inclusions between closed irreducible sets, hence actually equalities. This implies the proposition. We still need to prove the claim.
        \begin{proof}[Proof of the claim]\renewcommand{\qedsymbol}{}
            We need to construct a dense open subset of $C_1$ over which the fibres of $F$ are non-empty. Let us recall the variety $Z$ from Theorem \ref{binary:wellDefined}
            \[ Z = \{(M_1,M_2,d) \mid (M_1,M_2)\in S(C_1,C_2),\ d \in \Der_\Pi(M_1,M_2)\}.\]
            As shown in the proof of Theorem \ref{binary:wellDefined}, $Z$ is a vector bundle over $S(C_1,C_2)$. In addition, it is a subvariety of $C$ by interpreting elements as the module defined by $d$. By definition, $GL(V).Z = \mathcal{E}(S(C_1,C_2))$ is dense in $C$. Let $p_1:C \times C_2 \rightarrow C$ be the first projection. Then $p_1(S(C,C_2))\cap (GL(V).Z)$ is open in $GL(V).Z$ and non-empty. Since $p_1(S(C,C_2))$ is $GL(V)$-stable, $p_1(S(C,C_2))\cap Z$ is open in $Z$ and non-empty. Let $\pi_1: Z \rightarrow C_1$ be the first projection. Then
             \[ U = \pi_1 (p_1(S(C,C_2))\cap Z) \]
            is our desired open subset of $C_1$. Let $M_1\in U$. It remains to construct a preimage of $M_1$ under $F$. By the construction of $U$ there exists $(M,M_2)\in S(C,C_2)$ such that $p_1(M,M_2) = (M_1,M_2',d)\in Z$. Because $M_2$ and $M_2'$ are generic modules in $C_2$, they must both be the unique rigid module in $C_2$, in particular isomorphic. This allows us to define a short exact sequence
            \[ 0 \rightarrow M_2 \rightarrow M \rightarrow M_1 \rightarrow 0. \]
            Here the homomorphism of $T$-modules underlying the first map is the embedding $V_2\rightarrow V$ and the one underlying the second map is the projection $V \rightarrow V_1$. The sequence shows that $M_1\in \overline{F(X)}$, proving the claim.
        \end{proof}
    \end{proof}
    The previous proposition can be rephrased to give a cancellation property. 
    \begin{Corollary}
        Let $V = V_1 \oplus V_2$ be a direct sum of two locally free $T$-modules and let $C_1, C_1' \in \comp(V_1)$, $C_2\in \comp(V_2)$ be components where $C_2$ is rigid. We assume
        \[ C_1 * C_2 = C_1' * C_2.\]
        Then $C_1 = C_1'$.
    \end{Corollary}
    \begin{proof} This follows from the previous theorem:
        \[C_1 =  \faktor{(C_1 * C_2)}{C_2} = \faktor{(C_1' * C_2)}{C_2} = C_1'.\]
    \end{proof}
    The dual statements are:
    \begin{Proposition}
         Let $V = V_1 \oplus V_2$ be a direct sum of two locally free $T$-modules and let $C_1 \in \comp(V_1)$, $C_2\in \comp(V_2)$ be components with $C_1$ rigid. Then
         \[\mfaktor{C_1}{(C_1 * C_2)} = C_2.\]
    \end{Proposition}
    \begin{Corollary}
        Let $V = V_1 \oplus V_2$ be a direct sum of two locally free $T$-modules and let $C_1 \in \comp(V_1)$, $C_2,C_2'\in \comp(V_2)$ be components with $C_1$ rigid. We assume 
        \[ C_1 * C_2 = C_1 * C_2'.\]
        Then $C_2 = C_2'$.
    \end{Corollary}
    \begin{Example}[Leclerc's counterexample]
        Leclerc's counterexample is the smallest non rigid component in type $A$. It reveals a number of subtleties in these constructions. Let $C$ be of type $A_5$ with minimal symmetrizer. There is a $\mathbb{P}^1$-indexed family of $\Pi$-modules $M_{[\lambda:\mu]}$ of rank vector $\mathbf{d} = (1,2,2,2,1)$ given by 
        
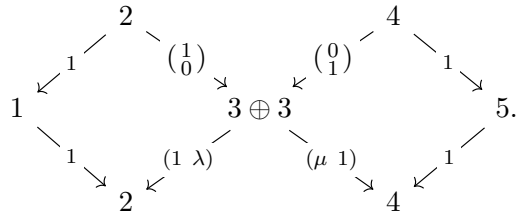
\[\begin{tikzcd}
        	& 2 && 4 \\
        	1 && 3\oplus3 && 5. \\
        	& 2 && 4
        	\arrow["1"{description}, from=1-2, to=2-1]
        	\arrow["{\left(\begin{smallmatrix}1 \\ 0\end{smallmatrix}\right)}"{description}, from=1-2, to=2-3]
        	\arrow["{\left(\begin{smallmatrix}0 \\ 1\end{smallmatrix}\right)}"{description}, from=1-4, to=2-3]
        	\arrow["1"{description}, from=1-4, to=2-5]
        	\arrow["1"{description}, from=2-5, to=3-4]
        	\arrow["{(\mu \ 1)}"{description}, from=2-3, to=3-4]
        	\arrow["{(1 \ \lambda)}"{description}, from=2-3, to=3-2]
        	\arrow["1"{description}, from=2-1, to=3-2]
        \end{tikzcd}\]
        They are indecomposable and pairwise non-isomorphic. They each have a three dimensional endomorphism ring but only two dimensional homomorphism spaces between non-isomorphic modules. Their orbits are 11 dimensional. Therefore, the family forms a 12 dimensional subvariety $C$ of $R^C_\Pi(\mathbf{d})$, which itself has dimension 12. This is the smallest non rigid component in $\comp$. Notably, $S(C,C)$ does not intersect the diagonal of $C \times C$. Despite every element of the family having non trivial self-extensions, the component still satisfies $\ext^1_\Pi(C,C)=0$. So we get $C * C = \overline{C \oplus C}$. On the other hand, for every $[\lambda:\mu]\in \mathbb{P}^1$ there is a short exact sequence
        \[0 \rightarrow M_{[\lambda:\mu]}\rightarrow \Pi e_2 \oplus \Pi e_4 \rightarrow M_{[\lambda:\mu]}\rightarrow 0.\]
        Let $C'$ be the component corresponding to the rigid module $\Pi e_2 \oplus \Pi e_4$. The short exact sequence implies $\faktor{C'}{C} = C = \mfaktor{C}{C'}$. So the division is defined but not inverse to $*$. Furthermore $\faktor{\overline{C \oplus C}}{C}$ is not defined. 
    \end{Example}
    \subsection{Change of symmetrizer}
    It turns out that the set $\comp$ only depends on $C$, not on the symmetrizer $D$. This was shown in \cite{GLS4}. Thus one might hope that the binary operation is independent of $D$, too. We will show this in the symmetric case. The symmetrizable case is open.\\
    For this section, we assume that $C$ is symmetric and connected. Then its minimal symmetrizer is $D = \diag(1,\dots,1)$ and every other symmetrizer is of the form $nD$ for some natural number $n$. We will compare $D$ and $nD$ for a fixed $n\in\mathbb{N}$. Objects defined using $D$ will be denoted by $-(1)$ and objects defined using $nD$ will be denoted by $-(n)$. So $\Pi(1) = \Pi(C,D,\Omega)$ and $\Pi(n) = \Pi(C,nD,\Omega)$. Note that $\Pi(1)$ is the preprojective algebra of a quiver, the varieties $R^C_{\Pi(1)}(V)$ are the nilpotent varieties studied by Lusztig \cite{LusztigSemicanonical} and the binary operation on $\comp_{(1)}$ is the one from \cite{lapid1}. Some of the considerations in this part are based on \cite{GLS2}.
    \begin{Definition}
        Consider the element
        \[\varepsilon = \sum_{i\in Q_0} \varepsilon_i \in \Pi(n).\]
        Let $M$ be a $\Pi(n)$-module. The reduction of $M$ is the $\Pi(1)$-module 
        \[\overline{M} = M/\varepsilon M.\]
        This defines a functor $R:\modulecat(\Pi(n))\rightarrow\modulecat(\Pi(1))$.
    \end{Definition}
    We collect some important properties of $R$.
    \begin{Lemma}$\ $\label{binary:redexact}
        \begin{enumerate}[label=(\roman*)]
            \item Let $M$ be a locally free $\Pi(n)$-module. Then $\overline{M}$ is a locally free $\Pi(1)$-module of the same rank vector.
            \item Let 
            \[0\rightarrow M_1 \rightarrow M \rightarrow M_2 \rightarrow 0\]
            be a short exact sequence of locally free $\Pi(n)$-modules. Then 
            \[0\rightarrow \overline{M_1} \rightarrow \overline{M} \rightarrow \overline{M_2} \rightarrow 0\]
            is a short exact sequence of locally free $\Pi(1)$-modules.
            \item Let $M$ be a crystal $\Pi(n)$-module. Then $\overline{M}$ is a crystal $\Pi(1)$-module.
        \end{enumerate}
    \end{Lemma}
    \begin{proof}
        The claims $(i)$ and $(ii)$ can be checked at each vertex independently. There, they follow from the representation theory of $K[X]/(X^n)$. Details can be found in Proposition 2.2 in \cite{GLS2}. Note that in (i) the claim '$\overline{M}$ is locally free' is empty because all $\Pi(1)$-modules are locally free. The rank vector is then just the dimension vector.\\
        For claim $(iii)$ we only need to show that $\overline{M}$ is $E$-filtered because over $\Pi(1)$ 'crystal', '$E$-filtered' adn 'nilpotent' are equivalent. We pick a filtration which witnesses that $M$ is $E$-filtered. By $(ii)$ the reduction of that filtration witnesses that $\overline{M}$ is $E$-filtered.
    \end{proof}
    This functor induces a morphism between the representation varieties of $\Pi(n)$ and $\Pi(1)$.
    \begin{Definition}
        Let $V$ be a locally free $T(n)$-module. We define a morphism
        \begin{align*}
            R: R^C_{\Pi(n)}(V)&\rightarrow R^C_{\Pi(1)}(V/\varepsilon V)\\
            M &\mapsto \overline{M}.
        \end{align*}
    \end{Definition}
    This morphism will allow a direct comparison between the binary operations $*_{(n)}$ and $*_{(1)}$. The varieties $R^C_{\Pi(n)}(V)$ and $R^C_{\Pi(n)}(V/\varepsilon V)$ have the same number of irreducible components because both are parametrised by the same subset of $B(-\infty)$. We are going to show that $R$ is surjective. Then the images of the irreducible components of $R^C_{\Pi(n)}(V)$ will be the irreducible components of $R_{\Pi(1)}(V/\varepsilon V)$. The following construction will give a section of $R$.
    \begin{Construction}
        Let $V$ be a locally free $T(1)$-module and $M\in R^C_{\Pi(1)}(V)$. We consider the locally free $T(1)$-module $M^{\oplus n}$ and define a locally free $\Pi(n)$-module structure $\widetilde{M}$ on a locally free $T(n)$-module $\widetilde{V}$ by replacing the $\varepsilon$-action on $M^{\oplus n}$ with the action \[\varepsilon(m,i) = \begin{cases}
            (m,i+1) &\textnormal{ for } m\in M, 1\leq i \leq n-1,\\
            0 &\textnormal{ for } m\in M, i = n.
        \end{cases}\]
    \end{Construction}
    The module $\widetilde{M}$ satisfies the nilpotency and commutativity relations by construction. The mesh relation does not involve the $\varepsilon_i$ because $C$ is symmetric. So $\widetilde{M}$ satisfies it since each copy of $M$ satisfies it. This construction does not work in the symmetrizable case. Let $M$ be a $\Pi(1)$-module. To ensure that $\widetilde{M}$ is a preimage of $M$ under $R$, we need to check that it lies in $R^C_{\Pi(n)}(\widetilde{V})$. This is done in the next lemma.
    \begin{Lemma}
        Let $M$ be a crystal module over $\Pi(1)$. Then $\widetilde{M}$ is a crystal module over $\Pi(n)$.
    \end{Lemma}
    \begin{proof}
        The tilde construction is exact and turns $E_i(1)$ into $E_i(n)$. Hence a filtration of $M$ with subquotients isomorphic to $E_i(1)$ induces a filtration of $\widetilde{M}$ with subquotients isomorphic to $E_i(n)$. This shows that $\widetilde{M}$ is $E$-filtered. By construction we have
        \begin{align*}
            &\widetilde{\sub_i(M)} = \sub_i(\widetilde{M})& \widetilde{Q_i(M)} \cong Q_i(\widetilde{M})\\
            &\widetilde{\fac_i(M)} \cong \fac_i(\widetilde{M})& \widetilde{K_i(M)} = K_i(\widetilde{M}).
        \end{align*}
        In particular, $\widetilde{M}$ is a crystal module by induction.
    \end{proof}
    \begin{Lemma}
        Let $M$ be a $\Pi(1)$-module. We identify $\overline{\widetilde{V}}$ and $V$. Then we have
        \begin{enumerate}[label=(\roman*)]
            \item $\widetilde{M} \in R^C_{\Pi(n)}(\widetilde{V})$,
            \item $\overline{\widetilde{M}} = M$,
            \item the map $R:R^C_{\Pi(n)}(\widetilde{V})\rightarrow R^C_{\Pi(1)}(V)$ is surjective,
            \item the map $R$ induces a bijection $\comp_{(n)}\rightarrow \comp_{(1)}, \ C\mapsto \overline{R(C)}$.
        \end{enumerate}
    \end{Lemma}
    \begin{proof}
    \begin{enumerate}[label=(\roman*)]
            \item This is the previous lemma.
            \item This follows from the construction of $\widetilde{M}$.
            \item We have shown that for any $M\in R_{\Pi(1)}(V)$ the module $\widetilde{M}$ is a preimage under $R$.
            \item This follows from $R$ being a surjective morphism between two varieties with the same number of irreducible components.
        \end{enumerate}
    \end{proof}
    Now we can prove independence from the symmetrizer.
    \begin{Theorem}
        Let $C$ be symmetric. Let $C_1, C_2\in \comp_{(n)}$ be components. Then
        \[\overline{R(C_1 *_{(n)} C_2)} = \overline{R(C_1)}*_{(1)} \overline{R(C_2)}.\]
    \end{Theorem}
    \begin{proof}
        We consider the set
        \[ S := \{(M_1,M_2) \in S(C_1,C_2) \mid (\overline{M_1},\overline{M_2}) \in S(\overline{R(C_1)},\overline{R(C_2)})\}\]
        of generic pairs whose reductions are generic. It is an open subset of $C_1\times C_2$, because it can be written as
        \[ S = S(C_1,C_2) \cap (R\times R)^{-1}(S(\overline{R(C_1)},\overline{R(C_2)})).\] By Theorem \ref{binary:wellDefined} we have
        \[C_1 *_{(n)} C_2 = \overline{\mathcal{E}(S)}. \]
        We consider an extension in $\mathcal{E}(S)$
        \[0\rightarrow M_2\rightarrow M \rightarrow M_1 \rightarrow 0.\]
        By Lemma \ref{binary:redexact} and the definition of $S$ this induces an extension
        \[0\rightarrow \overline{M_2}\rightarrow \overline{M} \rightarrow \overline{M_1}\rightarrow 0\]
        in $\mathcal{E}(S(\overline{R(C_1)},\overline{R(C_2)}))$. This shows $R(\mathcal{E}(S))\subset \mathcal{E}(S(\overline{R(C_1)},\overline{R(C_2)}))$ and by taking closures
        \[\overline{R(C_1 *_{(n)} C_2)} =\overline{R(\overline{\mathcal{E}(S)})}\subset \overline{S(\overline{R(C_1)},\overline{R(C_2)})}=\overline{R(C_1)}*_{(1)} \overline{R(C_2)}\]
        Equality holds everywhere because $\overline{R(C_1 *_{(n)} C_2)}$ is an irreducible component.
    \end{proof}
    As a consequence one can show that the bijection $R$ is compatible with the parametrisations by $B(-\infty)$.

\bibliographystyle{alpha}
\bibliography{references}

\end{document}